\documentclass{article}

\usepackage{arxiv}

\usepackage{amsmath,amsfonts,amssymb}
\usepackage{tabularx,booktabs}
\usepackage{colortbl}
\usepackage{mathtools}

\usepackage[textsize=small,textwidth=3.3cm]{todonotes}

\usepackage[utf8]{inputenc} \usepackage[T1]{fontenc}    \usepackage{hyperref}       \usepackage{url}            \usepackage{booktabs}       \usepackage{amsfonts}       \usepackage{nicefrac}       \usepackage{microtype}      \usepackage{lipsum}         \usepackage{graphicx}
\usepackage{doi}
\usepackage{xspace}
\usepackage[linesnumbered,ruled,vlined]{algorithm2e}

\SetCommentSty{AlgoCommentStyle}

\usepackage{tikz}
\usepackage{tikz-3dplot}
\usepackage{tikzscale}
\usetikzlibrary{matrix,chains,scopes,positioning,arrows,fit}

\usepackage{array}
\usepackage{amsthm}

\newtheorem{lemma}{Lemma}
\newtheorem{example}{Example}
\newtheorem{definition}{Definition}

\newcommand{\lincons}[1]{a_{#1}}
\newcommand{\rhs}{b}

\newcommand{\slack}[2]{\mathit{slack}({#1,#2})}
\newcommand{\reasoncon}{C_{\text{reason}}}
\newcommand{\conflictcon}{C_{\text{confl}}}
\newcommand{\learnedcon}{C_{\text{learn}}}

\newcommand{\NN}{\mathcal{N}}

\newcommand{\Z}{\mathbb{Z}}
\newcommand{\Zpos}{\mathbb{Z}_{\ge 0}}
\newcommand{\Zspos}{\mathbb{Z}_{>0}}

\newcommand{\literal}[1]{\ell_{#1}}

\newcommand{\var}[1]{x_{#1}}
\newcommand{\nvar}[1]{\overline{x}_{#1}}
\newcommand{\lit}{\ell_{r}}
\newcommand{\nlit}{\bar{\ell}_{r}}

\newcommand{\divides}{\mid}
\newcommand{\notdivides}{\nmid}

\newcommand{\seq}[1]{[1,\ldots,#1]}

\newcommand{\ie}{i.e.,\xspace}
\newcommand{\eg}{e.g.,\xspace}

\newcommand{\solver}[1]{\textsc{#1}\xspace}

\newcommand{\scipversion}{8.0.3}

\newcommand{\scip}{\solver{SCIP}}

\newcommand{\soplex}{\solver{SoPlex}}

\newcommand{\miplib}{\textsc{MIPLIB}\xspace}

\newcommand\vfloor[1]{\left\lfloor#1\right\rfloor}
\newcommand\vceil[1]{\left\lceil#1\right\rceil}

\newcommand{\fa}{\text{ for all }}

\definecolor{tabcolor}{HTML}{6666AA}
\definecolor{f1}{HTML}{000060}
\definecolor{f2}{HTML}{0000FF}
\definecolor{f3}{HTML}{36648B}
\definecolor{f4}{HTML}{4682B4}
\definecolor{f5}{HTML}{5CACEE}
\definecolor{f6}{HTML}{00FFFF}
\definecolor{f7}{HTML}{00DD99}
\definecolor{f8}{HTML}{008888}
\definecolor{f9}{HTML}{000000}

\usepackage{cleveref}        
\title{Improving Conflict Analysis in MIP Solvers by Pseudo-Boolean
Reasoning}

\author{ \href{https://orcid.org/0000-0003-0964-9802}{\includegraphics[scale=0.06]{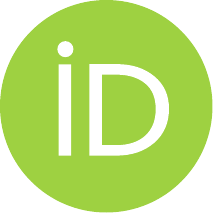}
\hspace{1mm}Gioni Mexi}\thanks{Zuse Institute Berlin, Germany} \\
	\texttt{mexi@zib.de} \\
\And
	\href{https://orcid.org/0000-0002-6320-8154}{\includegraphics[scale=0.06]{orcid.pdf}\hspace{1mm}Timo Berthold} \thanks{
	Fair Isaac Deutschland GmbH and TU Berlin, Germany} \\
	\texttt{timoberthold@fico.com} \\
	\And
	\href{https://orcid.org/0000-0003-0391-5903}{\includegraphics[scale=0.06]{orcid.pdf}\hspace{1mm}Ambros Gleixner} \thanks{
	HTW Berlin and Zuse Institute Berlin, Germany} \\
	\texttt{gleixner@htw-berlin.de} \\
	\And
	\href{https://orcid.org/0000-0002-2700-4285}{\includegraphics[scale=0.06]{orcid.pdf}\hspace{1mm}Jakob Nordström}\thanks{
	University of Copenhagen, Denmark and Lund University, Sweden} \\
	\texttt{jn@di.ku.dk} \\
}

\renewcommand{\headeright}{}
\renewcommand{\undertitle}{}
\renewcommand{\shorttitle}{Improving Conflict Analysis in MIP Solvers by Pseudo-Boolean
Reasoning}

\hypersetup{
pdftitle={A template for the arxiv style},
pdfsubject={q-bio.NC, q-bio.QM},
pdfauthor={Gioni Mexi, Timo Berthold, Timo Berthold, Timo Berthold},
pdfkeywords={First keyword, Second keyword, More},
}

\begin{document}
\maketitle

\begin{abstract}
  Conflict analysis has been successfully generalized from Boolean
  satisfiability (SAT) solving to mixed integer programming (MIP)
  solvers, but although MIP solvers operate with general linear
  inequalities, the conflict analysis in MIP has been limited to
  reasoning with the more restricted class of clausal constraint.
  This is in contrast to how conflict analysis is performed in
  so-called pseudo-Boolean solving, where solvers can reason directly
  with \mbox{$0$--$1$} integer linear inequalities rather than with
  clausal constraints extracted from such inequalities.
  
  In this work, we investigate how pseudo-Boolean conflict analysis
  can be integrated in MIP solving, focusing on \mbox{$0$--$1$}
  integer linear programs (\mbox{$0$--$1$} ILPs).  Phrased in MIP
  terminology, conflict analysis can be understood as a sequence of
  linear combinations and cuts. We leverage this perspective to design
  a new conflict analysis algorithm based on mixed integer rounding
  (MIR) cuts, which theoretically dominates the state-of-the-art
  division-based method in pseudo-Boolean solving.
  
  We also report results from a first proof-of-concept implementation
  of different pseudo-Boolean conflict analysis methods in the
  open-source MIP solver \solver{SCIP}.  When evaluated on a large and
  diverse set of \mbox{$0$--$1$} ILP instances from \miplib2017,
  our new MIR-based conflict analysis outperforms both previous
  pseudo-Boolean methods and the clause-based method used in MIP.
  Our conclusion is that pseudo-Boolean conflict analysis in MIP is a
  promising research direction that merits further study, and that it
  might also make sense to investigate the use of such conflict
  analysis to generate stronger no-goods in constraint programming. 
\end{abstract}

\keywords{Integer programming \and pseudo-Boolean reasoning \and conflict analysis \and cutting planes proof system \and mixed integer rounding \and division \and saturation}

\section{Introduction}
\label{sec:introduction}

The area of Boolean satisfiability (SAT) solving has witnessed
dramatic performance improvements over the last couple of decades, and
several techniques from SAT have also inspired developments for other
combinatorial optimization paradigms such as SAT-based and (linear)
pseudo-Boolean optimization, constraint programming, and mixed
integer programming.
In particular, conflict analysis as introduced in the works on
\emph{conflict-driven clause learning (CDCL)}~\cite{BS97UsingCSP,MS99Grasp,MMZZM01Engineering}
ushering in the modern SAT solving revolution has been picked up and
generalized in different ways
to these more general settings.
Interestingly, precursors of this version of conflict analysis and nonchronological backtracking
can be traced back all the way to early work in the AI community~\cite{stallman1977forward},
and related ideas have been used in constraint programming for
decades~\cite{ginsberg1993dynamic,jiang1994nogood}.  Our focus in this
paper is on conflict analysis in mixed integer programming and
pseudo-Boolean optimization, which we proceed to discuss next.

\subsection{Mixed Integer Programming
  and Conflict Analysis
}
\label{sec:clause-based-conflict-analysis}

The core method of mixed integer programming (MIP) is that a linear
programming (LP) relaxation of the problem is fed to an LP solver.
If the LP solver finds a solution that assigns real values to
integral variables, then either additional cut constraints can be
generated that eliminate such solutions, or the problem can be split
into subproblems by branching on integer variables, generating new
nodes in the search tree.
During the solving process infeasible nodes in the search tree are
pruned.  Unlike in SAT, there can be different reasons for
backtracking due to
infeasibility of the  LP relaxation,
node presolving  (propagation),
or to the current objective value of the relaxed problem being
worse than the best solution found so far
(branch-and-bound).
MIP solvers employ a multitude of further techniques
such as
symmetry detection, disjoint
subtree detection, restarts, et cetera. For a comprehensive
description of  MIP solving we refer the reader to,
\eg~\cite{achterberg2007constraint}.

The use of SAT techniques in MIP solvers has been a fruitful direction
of research over the last decades. Specifically,  CDCL conflict analysis
has proven to be a useful tool to enhance the performance of MIP
solvers by learning constraints from infeasibilities detected by
propagation or from the LP relaxation~\cite{achterberg2007conflict,
  sandholm2006nogood, WitzigBertholdHeinz2019}.
However, SAT and MIP solvers differ fundamentally in how they explore
the search space, in that SAT solvers search depth-first, maintaining
only the current state of the search, whereas in MIP the search tree
is generated in a ``best-first'' manner based on careful analysis on
search statistics such as dual bounds and integrality of LP solutions
to subproblems. 
These differences make it harder for MIP solvers to profit from conflict
analysis, and so in contrast to SAT solving, for which this technique
is absolutely crucial, in MIP solving it plays more of a supplemental
if still highly valuable role.

Although the setting is different, the
\emph{graph-based conflict analysis}~\cite{achterberg2007conflict}
used to learn from infeasibilities in MIP is very similar to the
classic SAT approach.
First, a partial assignment
is extracted that 
consists of branching decisions and implications that led to the
infeasibility. If the LP relaxation is infeasible, 
the information which bound changes led to infeasibility
is gathered from the non-zero duals of the LP.
Next, a directed acyclic graph is constructed that encodes information 
about the conflict, in that source nodes correspond to branching
decisions, non-source nodes encode implications, and the sink node
represents the infeasibility. Each cut in this graph that separates
the source nodes from the sink is a valid constraint. It is important
to note that all implications correspond to clausal constraints,
and so this conflict analysis operates not on the linear constraints
of the problem but on clauses extracted from these linear constraints.
(There are also methods that can learn general linear constraints from
infeasibilities, one notable example being
\emph{dual-proof analysis}~\cite{WitzigBertholdHeinz2019},
but this technique is limited to conflicts arising from infeasibility
of the LP relaxation and does not analyze or strengthen the partial assignment that led to infeasibility.)

\subsection{Pseudo-Boolean Solving
  and Conflict Analysis
}

Pseudo-Boolean (PB) solving is another approach specific to integer
linear programs with only binary variables, or \mbox{$0$--$1$} ILPs,
which are referred to as (linear) pseudo-Boolean formulas in the
PB solving literature. While MIP solvers find real-valued solutions and
try to push such solutions closer and closer to integrality, PB
solvers follow the SAT approach of considering only Boolean
assignments and trying to extend partial assignments
to more and more variables
without violating
any constraints. Just as in SAT, this search is performed in a
depth-first manner.

Some PB solvers stick very closely to SAT in that they
immediately translate the  \mbox{$0$--$1$} ILP into
conjunctive normal form (CNF) using auxiliary variables and then
run a standard CDCL SAT solver
\cite{ES06TranslatingPB,MML14Open-WBO,SN15Construction}.
Another approach, which is what is of interest in the context of this
work, is to extend the solvers to reason natively with
\mbox{$0$-$1$} linear inequalities
\cite{CK05FastPseudoBoolean,SS06Pueblo,LP10Sat4j,EN18RoundingSat}.
Such \emph{conflict-driven pseudo-Boolean solvers}
have the potential to run exponentially faster than CDCL-based solvers,
since their conflict analysis method is exponentially stronger than
that used in CDCL SAT solvers.

\providecommand{\varx}{x}
\providecommand{\pbcons}{C}
\providecommand{\cutreason}{R^*}

Since it is crucial for our work to understand the differences between
conflict analysis in MIP and PB solvers,  let us try
to provide a somewhat simplified exposition of PB solving in a
language that is meant to convey a MIP perspective (and where what
follows below is heavily indebted to~\cite{DGN21LearnToRelax}).
During the search phase, the pseudo-Boolean solver always first tries
to extend the current partial solution with any variable assignments that are propagated by
some linear inequality. When no further propagations are possible, the
solver  chooses some unassigned variable and makes a \emph{decision}
to assign this variable $0$ or~$1$, after which it again turns to
propagation. 
This cycle of decisions and propagations repeats until either a
satisfying assignment is found or some $0$--$1$ linear inequality $\pbcons$ is
violated.
In the latter case, the solver switches to the conflict analysis phase,
which works as follows:
\begin{enumerate}

\item
  The linear inequality~$R$ responsible for 
  propagating the last variable~$\varx$ 
  in~$\pbcons$
  to  the ``wrong value''
  from the point of view of~$\pbcons$
  is identified;
  this inequality~$R$ is referred to as 
  the \emph{reason constraint} for~$\varx$.

\item
  A
  \emph{division} or \emph{saturation} rule is
  applied to~$R$ to generate a
modified inequality~$\cutreason$  that
  propagates~$\varx$ tightly to its assigned value
  even when considered over the reals.

\item 
  A new linear constraint~$D$ is computed as the smallest integer
  linear combination of $\cutreason$ and~$\pbcons$ for which the
  variable~$\varx$ cancels and is eliminated. It is not too hard to
  show that it follows from the description above that 
  this constraint~$D$ is violated by the current partial
  assignment of the solver with the value of~$\varx$ removed,
and we can set
$\pbcons := D$ 
and go to step~$1$ again. 
\end{enumerate}
This  continues until a termination criterion
analogous to the
\emph{unique implication point (UIP)} notion used in SAT solving
leads to $D$~being declared as the learned constraint.
At this point, the solver undoes further assignments in reverse
chronological order until~$D$ is no longer violated, and then
switches back to the search phase.
We refer the reader
to the chapter~\cite{BN21ProofCplxSATplusCrossref}
for a more
detailed description of conflict-driven pseudo-Boolean solving
(and to the handbook~\cite{HandbookSAT21} for
an in-depth treatment of SAT and related topics in general).

In contrast to MIP conflict analysis, the algorithm described above
is not phrased in terms of the conflict graph,
but focuses on the syntactic
\emph{resolution}
method
\cite{Blake37Thesis,DP60ComputingProcedure,DLL62MachineProgram,Robinson65Machine-oriented}
employed in CDCL conflict analysis
and harnesses the observation by 
Hooker~\cite{Hooker88Generalized,Hooker92Generalized}
that resolution can be understood as a cut rule and extended to
\mbox{$0$--$1$} integer linear inequalities.
The conflict-graph-based analysis in MIP does not
operate on the reason constraints~$R$ as described above, but instead
on disjunctive clauses extracted from these constraints.
It is not hard to prove formally (appealing
to~\cite{BKS04TowardsUnderstanding,CCT87ComplexityCP,Haken85Intractability})
that this incurs an exponential loss in reasoning power compared to
performing derivations on the linear constraints themselves.

In practice, however, it seems fair to say that current pseudo-Boolean
solvers do not quite deliver on this promise of exponential gains in
performance. Although there are specific problem domains where
PB solvers outperform even commercial MIP
solvers~\cite{LBDEN20VerifyingMultiplication,SDNS20LearnedBNN},
evaluations over larger sets of benchmarks
\cite{berthold2008solving,DGN21LearnToRelax,DGDNS21Cutting}
have demonstrated that
the open-source MIP solver \scip~\cite{bestuzheva2021scip}
tends to be clearly more effective in solving
pseudo-Boolean optimization problems, and is also quite competitive
for decision problems. This is especially so for some decision problems
that are in some sense close to LP-infeasibility---such problems are
almost trivial for MIP solvers, but can be extremely challenging for
pseudo-Boolean solvers~\cite{EGNV18CombinatorialBenchmarks}.

\subsection{Questions Studied in This Work and
  Our Contributions}

Since mixed integer programming solvers and pseudo-Boolean solvers
approach \mbox{$0$--$1$} integer linear problems from quite different
angles, and seem to have complementary performance profiles, it is
natural to ask whether techniques from one of the paradigms can
be used to improve solvers based on the other paradigm.

Some MIP-inspired approaches
have been integrated with success in SAT and PB solvers,
perhaps most recently in~\cite{DGN21LearnToRelax},
where the PB solver 
\solver{RoundingSAT}~\cite{EN18RoundingSat}
makes careful use of the LP solver
\soplex~\cite{bestuzheva2021scip}
to detect infeasibility of LP relaxations and generate cut constraints
(though this paper also  raises many questions that would seem to
merit further study). 
However, in the other direction we are not aware of any work trying to
harness state-of-the-art techniques  from pseudo-Boolean solving
to improve the performance of MIP solvers.

In this work, we consider how the clausal conflict analysis in MIP
solvers can be replaced by pseudo-Boolean reasoning, focusing on 
\mbox{$0$--$1$} integer linear programs.
A key difference between the clausal and pseudo-Boolean conflict
analysis methods is that in the latter algorithm the linear reason
constraint~$R$  propagating a variable assignment might  need to be
modified, or \emph{reduced}, 
to another constraint~$\cutreason$ that propagates tightly also over the reals 
(which is already guaranteed to hold if~$R$ is a clausal constraint).
Viewed from a MIP perspective,
this \emph{reduction} step deriving~$\cutreason$ from~$R$
can be seen to be an application of one of two specific cut rules,
where saturation-based reduction
as in~\cite{LP10Sat4j}
corresponds to coefficient tightening
and division-based reduction
as in~\cite{EN18RoundingSat}
uses Chv\'atal-Gomory cuts.

This observation raises the question of whether more general cuts could also
used to obtain other, and potentially more powerful, reduction methods
for pseudo-Boolean conflict analysis.
The answer turns out to be yes, and, in particular,
we introduce a new reduction algorithm utilizing 
mixed integer rounding (MIR)
cuts~\cite{Gomory1960TR,marchand2001aggregation}.
A theoretical comparison of the MIR-based reduction rule with the
reduction methods currently used in PB solvers show that MIR-based
reduction dominates the division-based method that is considered to be
state of the art in pseudo-Boolean solving,
while saturation-based reduction and MIR-based reduction appear to be
incomparable.

We implement pseudo-Boolean conflict analysis for \mbox{$0$--$1$} ILPs
in the MIP solver \scip, including all three reduction methods
discussed above, and compare these different flavours of PB conflict
analysis with each other as well as with clausal MIP conflict analysis
on a large benchmark set consisting of pure \mbox{$0$--$1$} ILP
instances from \miplib2017. We find that the MIR-based pseudo-Boolean conflict
analysis has the best performance, beating not only the conflict
analysis methods in the PB literature but also the standard
clausal conflict analysis in \scip. Interestingly, the new method is
better measured not only in terms of number of nodes in the search
tree, but also in terms of the number of instances solved,
even though we only provide a proof-of-concept implementation lacking
many of the optimizations that would be included in a full integration
of this method into the \scip codebase.
Although our experimental data cannot provide conclusive evidence as
to what causes this improved performance, we observe that the
constraints learned from pseudo-Boolean conflict analysis seem more
useful in that they take part more actively in propagations 
than constraints obtained by clausal conflict analysis.

\subsection{Organization of This Paper} 

After reviewing preliminaries in Section~\ref{sec:preliminaries},
we give a detailed description of clausal and pseudo-Boolean conflict analysis
for \mbox{$0$--$1$} integer linear programs in
Section~\ref{sec:techniques},
including a discussion of the reduction methods found in the PB
literature and our new version using mixed integer rounding cuts,
and study
how the different reduction
rules compare in theory.
In 
Section~\ref{sec:experiments} we present our experimental results.
We conclude the paper in
Section~\ref{sec:conclusion}
by summarizing our work and discussing direction for future research.

\section{Preliminaries and Notation}
\label{sec:preliminaries}

Let $n \in \Zspos$, and $\NN := \seq{n}$.
We let $\var{i}$ denote Boolean (i.e., $\{0,1\}$-valued) variables
and 
$\literal{i}$ denote literals,
which can be either $\var{i}$ or its negation $\nvar{i} = 1 - \var{i}$. 
A \emph{pseudo-Boolean constraint} is a \mbox{$0$--$1$} integer linear inequality
\begin{equation}
    \label{eq:PB}
    \sum_{ i \in \NN} \lincons{i} \literal{i} \geq \rhs
    \ ,
\end{equation}
where we can assume without loss of generality that
$\lincons{i} \in \Zpos \fa i \in \NN$ and $\rhs \in \Zpos$
(so-called \emph{normalized form}).
We can convert ``$\leq$''-constraints with \mbox{$0$--$1$} variables to 
\mbox{``$\ge$''-constraints} by multiplying the constraint by $-1$
and normalizing, \ie replacing the variables by literals. Moreover,
equalities ``='' can be viewed as syntactic sugar for
two opposing inequalities, which can
also be transformed into normalized pseudo-Boolean format.  In particular, every pure
\mbox{$0$--$1$} integer linear program can be transformed to a normalized pseudo-Boolean
representation.
Note that in \cref{sec:techniques} we develop our theory and algorithms using normalized PB constraints for simplicity of exposition. However, in our actual implementation and experiments (described in \cref{sec:experiments}), we directly operate on general linear constraints.

A (partial) assignment $\rho$ is a (partial) map  from variables  to $\{0,1\}$,
which is extended to literals by respecting the meaning of negation.
We call a literal $\literal{i}$
\emph{falsified} or \emph{false}
if $\rho(\literal{i}) =  0$ and
\emph{satisfied} or \emph{true}
if $\rho(\literal{i}) =  1$. 
If $\rho$ is undefined for a literal, we call the literal \emph{unassigned} or \emph{free}.
A constraint is satisfied under some partial assignment $\rho$ if the respective inequality holds, independently of which values the unassigned literals take,
and is falsified if no assignment to the unassigned literals can make the inequality true.

The slack of a PB constraint $C: \sum_{ i \in \NN} \lincons{i} \literal{i} \geq
\rhs \,$ under a partial assignment $\rho$ is defined as $\slack{C}{\rho} :=
\sum_{ \{i \in \NN \, : \rho(i) \neq 0 \} } \lincons{i} - \rhs$.
With this definition, 
$C$
is falsified under $\rho$ if and only if $\slack{C}{\rho} < 0$. 
For example the constraint $C: 2\nvar{1} + 2\var{2}+ 3\var{3}\geq 4$ is falsified under
the partial assignment $\rho = \{ \var{1}= 1, \var{2} = 0 \} $ since
$\slack{C}{\rho} = -1 < 0$.
For a non-falsified constraint $C$ and an unassigned literal $\literal{i}$ with coefficient $\lincons{i}$, the constraint propagates $\literal{i}$  if and only if $\slack{C}{\rho} < \lincons{i}$.
For instance, the same constraint $C: 2\nvar{1} + 2\var{2}+ 3\var{3}\geq 4$ propagates both variables $\var{2}$ and $\var{3}$ to 1 under the partial assignment $\rho = \{ \var{1}= 1 \} $ since $\slack{C}{\rho} = 1$ is strictly smaller than the coefficients of each of the variables.
A constraint propagates the assignment of a free variable tightly if the slack under the current partial assignment is 0.
For any two pseudo-Boolean
constraints $C$ and $C'$ and partial assignment $\rho$ it holds
that the slack is subadditive, \ie $\slack{C+C'}{\rho} \le \slack{C}{\rho} +
\slack{C'}{\rho}$.
The decision level of a literal $\literal{i}$ under a partial assignment $\rho$
is the number of decisions prior to the fixing of $\literal{i}$.
Note that the first fixing in every decision level is a decision literal.

\section{Conflict Analysis Algorithms}
\label{sec:techniques}
For simplicity, in this section we
present all algorithms
in a pseudo-Boolean framework, where all coefficients and constants are integral, and the proofs of correctness that we provide also make crucial use of this fact. It is important to note that this is not the case in the actual implementation in SCIP, which operates with real-valued coefficients and constants. In fact, one of the challenges in implementing pseudo-Boolean conflict analysis in a MIP framework is that careful thought is required to rephrase the algorithms in such a way that they can deal with real-valued data but are still correct.
Next, we describe the details of conflict analysis algorithms used in PB solvers 
and the different techniques that we consider in this paper.

\subsection{Clausal Conflict Analysis}
\label{sec:clausebased}

To explain the idea of conflict analysis, we first consider the case
where all constraints are clauses.  Conflict analysis begins at the
stage where a conflict clause $\conflictcon$ is falsified by the
current partial assignment $\rho$.  Let $\lit$ be the literal in
$\conflictcon$ that was last propagated to false, and let $\reasoncon$
be the reason clause in chronological order that is responsible for
the propagation, i.e., we have $\conflictcon=C'\vee\lit$ and
$\reasoncon=C''\vee\nlit$.  Using the resolution rule, we can derive
the so-called \textit{resolvent} $C'\vee C''$ as a new learned clause
$\learnedcon$.

Note that, even after removing $\lit$ from the partial assignment
$\rho$, both $C'$ and $C''$ remain falsified: $C'$ because
$\conflictcon=C'\vee\lit$ and $\lit$ were false, and $C''$ because
$\reasoncon=C''\vee\nlit$ propagated.  This is the key invariant of
the algorithm: At any point during the algorithm the resolvent is
falsified by the remaining partial assignment $\rho$.

Hence, we can replace the conflict clause by the resolvent and
continue this process.  At each step either a propagating literal is
removed from $\rho$ or the learned clause is empty (at which point
unsatisfiability is proven) or the last fixed literal is a decision
literal.  In the third case, we have reached a \emph{first unique
implication point} (FUIP) and conflict analysis terminates, with the
final resolvent being the learned clause $\learnedcon$.  With
$\learnedcon$ added, propagation on the previous decision level will
prevent the last infeasible decisions to happen as the search continues.

It is straightforward to apply this algorithm to problems with
$0$--$1$ linear constraints.  Suppose $\sum_{
i \in \NN} \lincons{i} \literal{i} \geq \rhs$ is the initial conflict
constraint falsified under $\rho$, then $\bigvee_{i
: \lincons{i}>0 \land \rho(\ell_i) = 0} \ell_j$ can be used as initial
conflict clause.  Analogously, we can extract at each step a reason
clause from the linear constraint that propagated the last literal and
perform resolution. After terminating at an FUIP, the learned clause
can be added as linear constraint to the solver.

\subsection{PB Conflict Analysis}
As in the clausal version, the main idea of PB conflict analysis is also to find a new constraint that explains the
infeasibility of the current subproblem under a falsifying partial assignment.
\Cref{algo:CA} shows the base algorithm for all variants of PB conflict analysis considered in this paper, using the first
unique implication point (FUIP) learning scheme. 
It is initialized with a falsifying partial assignment $\rho$ and a conflicting constraint
$\conflictcon$ under $\rho$.
First, the learned conflict constraint $\learnedcon$ is set equal to the conflict constraint $\conflictcon$.
In each iteration, we extract the latest literal $\lit$ from $\rho$. 
If the literal assignment was due to propagation of a constraint and the negated literal $\nlit$ occurs in $\conflictcon$, then we extract the reason 
constraint $\reasoncon$ that propagated $\lit$. In line \ref{line:reduce} we ``reduce'' 
the reason constraint such that the resolvent of $\learnedcon$ and the reduced reason
$\reasoncon$ (Line \ref{line:resolve}) that cancel the last literal $\lit$ is still falsified under the remaining partial assignment $\rho$.
The conflict constraint is set to the resolvent and we continue until we reach an
FUIP ($\learnedcon$ is \textit{asserting}) or we prove  $\learnedcon$ makes the problem infeasible.
We have reached an FUIP if $\learnedcon$ would propagate some literal after removing at least all literal assignments in the current decision level from $\rho$.
We have shown that the problem is infeasible if $\learnedcon$ is falsified under an empty partial assignment $\rho$.
At this point, the learned constraint can be added to the constraint database of our problem to prevent the solver from exploring the same search space again.

\begin{algorithm}[h]
    \DontPrintSemicolon
	\SetKwInOut{Input}{Input}\SetKwInOut{Output}{Output}
	\SetKwInOut{Init}{Initialization}
    \SetKwFunction{reduce}{reduce}
    \SetKwFunction{resolve}{resolve}
    \SetArgSty{upshape}
    \SetKwFunction{pop}{pop}
    \SetKwFunction{reason}{reason}
	\Input{conflict constraint $\conflictcon$, falsifying partial assignment $\rho$}
	\Output{learned conflict constraint $\learnedcon$}
	$\learnedcon \leftarrow \conflictcon$ \;
    \While{$\learnedcon$ not asserting \textbf{and} $\learnedcon \neq  \, \perp$\label{line:startloopCA}}
	{   
        $\lit \leftarrow$ literal last assigned on $\rho$\;
        \If{ $\lit$ propagated \textbf{and} $\nlit$ occurs in \text{$\learnedcon$}}
        {
            $\reasoncon \leftarrow \reason(\lit, \rho )$\;
            $\reasoncon \leftarrow \reduce(\reasoncon,\learnedcon, \lit, \rho)$\label{line:reduce}\;
            $\learnedcon \leftarrow \resolve(\learnedcon,\reasoncon,\lit)$\label{line:resolve}\;
        }
        $\rho \leftarrow \rho \setminus \{\lit\}$\;
   }
 	\Return{$\learnedcon$}\;

	\caption{Pseudo-Boolean Conflict Analysis Algorithm \label{algo:CA}}
\end{algorithm}

The key invariant of \Cref{algo:CA} is that in each iteration the
resolvent $\learnedcon$ remains falsified. In the clausal version this
holds even without the reduction step in line \ref{line:reduce}.
However, for general linear constraints this is not the case, as shown
by the next example.

\begin{example}
   \label{ex:noconflict}
   Consider the two PB constraints $\reasoncon = \var{1} + \var{2} +2x_3 \ge 2$ and
$\conflictcon = \var{1} + 2\nvar{3} + \var{4} + \var{5} \ge 3$ and the partial assignment
   $\rho = \{\var{1} = 0, \var{3} = 1\}$ where $\var{1} = 0$ is a decision, and $\var{3} = 1$ is propagated by $\reasoncon$. Under $\rho$ the constraint $\conflictcon$ is
   falsified. Applying generalized resolution to cancel $\var{3}$ yields the
   constraint $2x_1 + \var{2} + \var{4} + \var{5} \ge 3$ which is not falsified under $\rho$.
\end{example}

In the following sections, we present three different reduction techniques for
\Cref{algo:CA} that operate directly on PB constraints. 
The main idea is to apply valid operations on the reason constraint to reduce the slack and ensure that the resolvent will have negative slack.
The two main ingredients of the reduction techniques are \textit{weakening} and \textit{cutting planes} and are applied to the reason constraint until the invariant is fulfilled.

Weakening a literal in a PB constraint simply sets it to 1. For example, weakening a constraint $C: \var{1} + \var{2} +2x_3 \ge 2$ on $\var{1}$ yields $\var{2} +2x_3 \ge 1$.
Weakening is a valid operation since it simply adds a multiple of the valid bound constraint 
$\nvar{1} \geq 0$ to $C$. Note that weakening entails a loss of information. However, as we will see, it is a necessary operation to reduce the slack of the reason constraint. Note that whenever weakening is applied on non-falsified literal, it does not change the slack of the constraint.
See \Cref{sec:weakening} for more details on weakening.

Our main focus in this paper, however, is the second necessary ingredient of the reduction algorithm: cutting planes (cuts). Cuts are applied to the ``weakened'' version of the reason constraint in order to reduce its slack.
We first present two well-documented cuts from existing literature, namely \textit{Saturation} (\Cref{sec:Saturation}) and \textit{
Division} (\Cref{sec:Division}).
Both ensure the reduction of the slack of the
reason constraint to 0 at least after weakening all non-falsified literals in the original reason constraint.
In \Cref{sec:MIR}, we introduce a new cut based on the \textit{Mixed Integer Rounding (MIR)}
procedure and prove that it has the same property.  In \Cref{sec:dominance} we show that the reduction algorithm using MIR always returns an equally strong or
stronger reason constraint than the reduction using Division.

\subsection{Saturation-based Reduction}
\label{sec:Saturation}

First, we present the \textit{Saturation} cut. Then, we provide details about the \textit{Saturation-based Reduction} algorithm and demonstrate how the reduction ensures that the key invariant of conflict analysis holds.

\begin{definition}[Saturation Cut]
    Let $C: \sum_{ i \in \NN} \lincons{i} \literal{i} \geq \rhs.$ The \textbf{Saturation Cut} of $C$ is given by the constraint 
    \label{def:Saturation}
            \[\sum_{i \in \NN } \min \{\lincons{i}, \rhs\}\literal{i}  \geq \rhs.\]
\end{definition}
Saturation is a valid cut known as coefficient
tightening cut in the MIP
literature~\cite{brearley1975analysis} and does not entail a loss of information.
\Cref{algo:ReduceSAT} is used to reduce the reason constraint $\reasoncon$ before applying
generalized resolution.
Similar to the implementation in \cite{CK05FastPseudoBoolean}, in each iteration, the algorithm picks a
non-falsified literal in the reason constraint different from the literal we
are resolving on and weakens it. Then it applies the Saturation cut to the
resulting constraint. The algorithm terminates when the slack of the resolvent
becomes negative.

\begin{algorithm}[h]
    \DontPrintSemicolon
	\SetKwInOut{Input}{Input}\SetKwInOut{Output}{Output}
	\SetKwInOut{Init}{Initialization}
    \SetKwFunction{weaken}{weaken}
    \SetKwFunction{saturate}{saturate}
    \SetKwFunction{resolve}{resolve}
	\Input{conflict constraint $\conflictcon$, reason constraint $\reasoncon$, \\ 
      literal to resolve $\lit$, partial assignment $\rho$}
	\Output{reduced reason $\reasoncon$}
   \While{$\slack{(\resolve(\reasoncon,\conflictcon,\lit))}{\rho}\geq 0$}
    {
        $\literal{j} \leftarrow$ non falsified literal in $ \reasoncon\backslash \{\lit\}$ \;
        $\reasoncon \leftarrow \weaken(\reasoncon,\literal{j})$ \;
	    
        $\reasoncon \leftarrow 
        \saturate(\reasoncon)$ \;
    }
 	\Return{$\reasoncon$}\;

	\caption{Saturation-based Reduction Algorithm \label{algo:ReduceSAT}}
\end{algorithm}

For completeness, we prove the following well-known fact that demonstrates that the slack of the reason constraint will be reduced to 0 at the latest after weakening the last non-falsified literal and applying the Saturation cut. Since the slack is subadditive, the resolvent's slack becomes negative and the resolvent is thus falsified.

\begin{lemma}
    \label{lem:slackzeroSAT}
     Let $\rho$ be a partial assignment, and $\reasoncon: \sum_{ i \in \NN} \lincons{i}
     \literal{i} \geq \rhs$ a constraint propagating a literal $\literal{r}$
     to 1. Further, assume that $\slack{\reasoncon}{\rho} > 0$. Then, 
     after weakening all non-falsified literals in $\reasoncon$ (except for $\literal{r}$) and applying
     Saturation on $\reasoncon$, the slack of the reduced reason constraint is 0.
\end{lemma}
\begin{proof}
    First, we rewrite the constraint $\reasoncon$ as
    \[\sum_{j: \rho(j) = 0 } \lincons{j}\literal{j} + \sum_{i\neq r: \rho(i) \neq 0 } \lincons{i}\literal{i} + \lincons{r}\literal{r}\geq \rhs. \]
    Since $\slack{\reasoncon}{\rho}  := \sum\limits_{i\neq r: \rho(i) \neq 0}\lincons{i} + \lincons{r} - \rhs > 0$, it holds that
    \begin{equation}
        \label{eq:slackreason}
        \lincons{r} > \rhs - \sum\limits_{i\neq r: \rho(i) \neq 0}\lincons{i}.
    \end{equation}
    After weakening all literals from $\{i\neq r: \rho(i) \neq 0\} $  the
    constraint $\reasoncon$ becomes 
    \begin{equation}
        \label{eq:weakercon}
        \sum\limits_{j: \rho(j) = 0 } \lincons{j} \literal{j} + \lincons{r}\literal{r}\geq \tilde{b} := \rhs - \sum\limits_{i\neq r: \rho(i) \neq 0 } \lincons{i}.
    \end{equation}
Applying Saturation on \eqref{eq:weakercon} sets  $\lincons{r}$ to $\tilde{\rhs}$ because of \eqref{eq:slackreason}. Therefore the slack of the reduced reason constraint becomes $\tilde{b} - \tilde{b} = 0$.
\end{proof}

\subsection{Division-based Reduction}
\label{sec:Division}
A very competitive alternative to Saturation in the reduction algorithm is based on Division cuts. 
Division is also a valid cut known as
Chv\'atal-Gomory cut in the MIP literature~\cite{chvatal1973edmonds}.

\begin{definition}[Division Cut]
    \label{def:cg}
    Let $C: \sum_{ i \in \NN} \lincons{i} \literal{i} \geq \rhs.$ The \textbf{Division Cut} of $C$ with divisor $d \in \Zspos$ is given by the constraint
    \[
    \sum_{i \in \NN } \vceil{\frac{\lincons{i}}{d}}\literal{i}  \geq \vceil{\frac{\rhs}{d}}.
    \]
\end{definition}
To see why this procedure is valid, we can think of it as three steps:
dividing by $d$ maintains the validity of the constraint;
rounding up coefficients on the left-hand side relaxes the constraint and
is hence valid;
the validity of rounding up the right-hand side follows from the integrality of the left-hand side coefficients and literals.

In the Division-based reduction algorithm, the divisor $d$ used is the coefficient of the literal $\literal{r}$ we are resolving on.
As proven in \cite{EN18RoundingSat}, it suffices to weaken non-falsified variables with a
coefficient that is not a multiple of $\lincons{r}$, i.e., from the index set
$W := \{ i\in \NN : \rho(i) \neq 0 \text{ and }  \lincons{i} \notdivides
\lincons{r} \}.$
After weakening all literals in $W$ and applying
Division on $\reasoncon$, the slack of the reduced reason constraint is 0, which for completeness we include in \Cref{lem:slackzeroDIV} below.

\subsection{MIR-based Reduction}
\label{sec:MIR}
Next, we define a new cut for the reduction
algorithm based on the \emph{Mixed Integer Rounding} formula~\cite{marchand2001aggregation}, which is a generalization of Gomory's mixed integer cuts~\cite{Gomory1960TR}.
\begin{definition}[Mixed Integer Rounding Cut]
    \label{def:miprounding}
    Let $C: \sum_{ i \in \NN} \lincons{i} \literal{i} \geq \rhs.$ The \textbf{Mixed Integer Rounding (MIR) Cut} of $C$ with divisor $d \in \Zspos$ is given by the constraint
    \begin{equation}
    \sum_{ i \in I_1} \vceil{\frac{\lincons{i}}{d}} \literal{i}+ \sum_{i \in I_2}\left(\vfloor{\frac{\lincons{i}}{d}}+\frac{f(\lincons{i}/d)}{f(\rhs/d)}\right) \literal{i} \geq \vceil{\frac{\rhs}{d}},
    \end{equation}
        where \[I_1 = \{i \in \NN \, : \, f(\lincons{i}/d)\geq f(\rhs/d) \text{ or } f(\lincons{i}/d) \in \Z\}, \] \[I_2 = \{i' \in \NN \, : \, f(\lincons{i'}/d) < f(\rhs/d) \text{ and } f(\lincons{i'}/d) \notin \Z\},\]
    and $f(\cdot) = \cdot - \vfloor{\cdot}$. To obtain a normalized version of the MIR cut, we multiply both sides of the constraint by $(\rhs \mod d)$.
\end{definition}
The proof that applying MIR to a constraint is a valid procedure can be found in
\cite{marchand2001aggregation}.
Similar to the Division-based reduction, it suffices to weaken non-falsified variables with a
coefficient that is not a multiple of $\lincons{r}$ before applying MIR in order
to reduce the slack of the reason constraint to at most 0. This is shown in the
following lemma.

\begin{lemma}
    \label{lem:slackzeroDIV}
    Let $\rho$ be a partial assignment and $\reasoncon: \sum_{ i \in \NN} \lincons{i}
    \literal{i} \geq \rhs$ a constraint propagating a literal $\literal{r}$
    to 1.
Then, 
    after weakening all non-falsified literal in $W := \{ i\in \NN : \rho(i) \neq 0 \text{ and }  \lincons{r} \notdivides
\lincons{i} \}$ and applying
    Division or MIR on $\reasoncon$ with $d=a_r$, the slack of the reduced reason is
    at most 0.
\end{lemma}
\begin{proof}
After weakening all literals in $W$, the constraint $\reasoncon$ becomes
    \begin{equation}
        \label{eq:weakerconDiv}
        \lincons{r}\literal{r} + \sum\limits_{j \in \NN \backslash W } \lincons{j} \literal{j} \geq \tilde{\rhs} := \rhs - \sum\limits_{i \in W } \lincons{i}.
    \end{equation}
    Its slack is
    \[
    \slack{\reasoncon}{\rho}
    = \lincons{r} + \sum\limits_{{\substack{j \in \NN \backslash W: \\ \rho(j)\neq 0}}} \lincons{j} - \tilde{\rhs}
    = \lincons{r} + \sum\limits_{\mathclap{\substack{j \in \NN: \\ \rho(j)\neq 0, a_r\divides a_j}}} \;\lincons{j} - \tilde{\rhs}.
    \]
    Since weakening does not affect the slack, we have $\slack{\reasoncon}{\rho}< \lincons{r}$.

1. After applying the Division cut to \eqref{eq:weakerconDiv} with $d=\lincons{r}$,
    the slack becomes
    \begin{equation}\label{eq:slackreasonDiv}
    \slack{\reasoncon}{\rho}
    = 1 + \sum\limits_{\mathclap{\substack{j \in \NN: \\ \rho(j)\neq 0, a_r\divides a_j}}} \;\vceil{\frac{\lincons{j}}{\lincons{r}}} - \vceil{\frac{\tilde{\rhs}}{\lincons{r}}}
    \leq 1 + \sum\limits_{\mathclap{\substack{j \in \NN: \\ \rho(j)\neq 0, a_r\divides a_j}}} \;\frac{\lincons{j}}{\lincons{r}} - \frac{\tilde{\rhs}}{\lincons{r}}
    < \frac{\lincons{r}}{\lincons{r}} = 1.
    \end{equation}
    Because $\reasoncon$ contains only integer coefficients after applying the
    division rule, its slack is integer; hence, it must be at most 0.

2. Applying the MIR cut to \eqref{eq:weakerconDiv} with
$d=\lincons{r}$ results in the same slack as
in \eqref{eq:slackreasonDiv}. This is because all left-hand side
coefficients in the slack computation are divisible by $d$, hence they
fall into the index set $I_1$ and are transformed the same way as by
the Division cut.
\end{proof}

\subsection{Dominance Relationships}
\label{sec:dominance}

In this section, we would like to discuss briefly known dominance relationships between the different reduction techniques. 
The ultimate goal is to find a reduction technique that yields the strongest possible reason constraint to use in the resolution step of conflict analysis.
The following lemma states the well-known fact that constraints from Saturation-based reduction are always at least as strong as the resolvents created during clausal conflict analysis as described in \Cref{sec:clausebased}.

\begin{lemma}
    \label{lem:dominancewgr}
    Let $\rho$ be a partial assignment and $\reasoncon: \sum_{ i \in \NN} \lincons{i} \literal{i} \geq \rhs$ be a PB constraint which propagates literal $\ell_{r}$ to 1. Let $\reasoncon'$ and $\reasoncon''$ be the constraints obtained by clausal and Saturation-based reduction, respectively. Then $\reasoncon''$ implies $\reasoncon'$.
    \end{lemma}
    
    \begin{proof}
    Under the current partial assignment $\rho$, the disjunctive clause reason is given by
    $\ell_{r} \bigvee_{j : \rho(\ell_j) = 0} \ell_j$,
    which can be linearized as
    \begin{equation*}
        \reasoncon' \,:\, \ell_{r} + \sum_{j \,: \, \rho(\ell_j) = 0 } \literal{j}  \geq 1.
    \end{equation*}
    Now let $W$ be the set of all non-falsified literals, except $\ell_{r}$.
    After weakening all literals in $W$ and applying Saturation, we obtain the constraint
    \begin{equation*}
        \reasoncon'' \,: \,\min\{\lincons{r}, \rhs - \sum_{i \in W} \lincons{i}\} \literal{r} + \sum_{j : \rho(l_j) = 0} \min\{\lincons{j}, \rhs - \sum_{i \in W} \lincons{i}\} \literal{j} \geq \rhs - \sum_{i \in W} \lincons{i}.
    \end{equation*}
    As in the proof of \Cref{lem:slackzeroSAT}, it holds that  $\min\{\lincons{r}, \rhs - \sum_{i \in W} \lincons{i}\} = \rhs - \sum_{i \in W} \lincons{i}$.
    Now, after scaling $\reasoncon'$ by $\rhs - \sum_{i \in W} \lincons{i}$ we see that $\reasoncon''$ has the same right-hand side as $\reasoncon'$, but smaller or equal coefficients on the left-hand side.
    \end{proof}

In~\cite{GNY19DivisionSaturation} the authors show
that using Division instead of Saturation can be
exponentially stronger, and that a single Saturation step
can be simulated by an exponential number of Division steps.

    The dominance of MIR cuts over Chvátal-Gomory cuts is a well-known fact in the MIP literature. The
    following lemma shows essentially the same result as in \cite{CORNUEJOLS20011}, but in
    the context of conflict analysis for pseudo-Boolean problems.

\begin{lemma}
    \label{lem:dominanceDivMir}
    Let $\rho$, $\reasoncon$, $\ell_{r}$ be given as in \Cref{lem:dominancewgr}. Let $\reasoncon'$
    and $\reasoncon''$ be the constraints obtained by Division-based and MIR-based reduction, respectively. Then $\reasoncon''$ implies $\reasoncon'$. 
\end{lemma}
\begin{proof}
    Let $\reasoncon', \reasoncon''$ be constraints as in \Cref{def:cg} and \ref{def:miprounding}, respectively, with divisor $d=\lincons{r}$. The constraints have the same right-hand side and the same coefficients for all literals $\ell_i$ with $i \in I_1$. 
For $i \in I_2$ the coefficient of literal $\ell_{i}$ in $\reasoncon'$ is given by $\vceil{\frac{\lincons{i}}{\lincons{r}}}$ and in $\reasoncon''$ by
$\vfloor{\frac{\lincons{i}}{\lincons{r}}}+\frac{f(\lincons{i}/\lincons{r})}{f(\rhs/\lincons{r})}$.
The coefficients of the literals $\literal{i}$ in $\reasoncon'$ are always
greater than or equal to the coefficients in $\reasoncon''$, since by definition of
the set $I_2$ it holds that $f(\lincons{i}/\lincons{r}) / f(\rhs/\lincons{r}) < 1$.
Therefore $\reasoncon''$ implies $\reasoncon'$. 
\end{proof}

As an example, consider the partial assignment $\rho = \{\var{1} = 0, \var{2}= 0, \var{3} = 1\}$ and the constraint
$\reasoncon: 2\var{1}+6\var{2}+10\var{3}\ge 8$ which propagates variable $\var{3}$ to 1. Then the Division cut with divisor $10$ is
$\var{1}+\var{2}+\var{3}\ge 1$. The MIR cut with the same divisior is
$\frac{0.2}{0.8}\var{1}+\frac{0.6}{0.8}\var{2}+\var{3}\ge 1$.
Multiplying with $8 \mod 10 = 8$ gives the normalized MIR cut $2\var{1}+6\var{2}+8\var{3}\ge 8$. 
The normalized MIR cut is stronger than the Division cut, which
can be easily seen after scaling the Division cut by $8$.

\subsection{Practical Aspects of Weakening}
\label{sec:weakening}

While the evaluation of different weakening strategies is not the focus of this paper, we would like to discuss briefly some practical aspects of weakening literals.
In our implementation we consider the following iterative weakening strategy: weaken free
literals first followed by implied literals.
We stop as soon as the
resolvent is falsified under the remaining partial assignment.
Intuitively, this order is motivated by the fact that free literals are not relevant for
the propagation of literals in the reason constraint and do not affect the
falsification of the conflict constraint.

However, the optimal order in which to weaken literals is not yet fully understood, and remains an open research question. Possible approaches include weakening literals in order of increasing or decreasing coefficient size.
In~\cite{le2020weakening} the authors conducted experiments with various weakening techniques, including partial weakening of literals and applying weakening on the conflict constraint, but the results did not yield a conclusive ``best'' weakening strategy.

A simple alternative is to weaken literals in a single sweep. For all three reduction
algorithms, we can weaken the entire candidate set of literals as stated in \Cref{lem:slackzeroSAT} and \Cref{lem:slackzeroDIV} at once.
Weakening literals all at once leads to a faster
reduction algorithm since repeated slack computations are avoided and only one cut is applied in each iteration. However, this may result in the constraint being less
informative due to unnecessary weakening of literals.

\section{Experiments}\label{sec:experiments}

It is well known in the SAT and PB communities that
efficient conflict-driven search requires substantial amounts of
very careful engineering. In this first work, our focus has been on
importing and adapting the pseudo-Boolean conflict analysis to a MIP
setting---which is a nontrivial task in its own right---leaving
further optimizations as future work.

All techniques from \Cref{sec:techniques} have been implemented in the open source MIP solver \scip \scipversion~\cite{bestuzheva2021scip} and 
we conducted extensive experiments to compare the different reduction techniques in isolation.
Obtaining accurate performance results for MIP solvers requires a carefully
designed experimental setup since even small changes to algorithms or the input data can
have a large impact on the behavior and the performance of the solver. This is a
well-known fact in the MIP literature known as \textit{performance variability}
\cite{lodi2013performance}. To lessen the effects of performance variability
and obtain a fair comparison of the different reduction techniques in the
context of MIP solving, we use a fairly large and diverse testset of
instances and different permutations of each instance, see, e.g.,~\cite{MIPLIB2017}.
Our experiments were carried out on all pure \mbox{$0$--$1$} models from the \miplib2017
collection~\cite{MIPLIB2017}. After removing numerically unstable models (with the tag
``numerics'') our testset consists of 195 instances permuted by 5 different random seeds, giving a total of 975 measurements per run.
For the remainder of this paper, we will refer to the combination of a model and a permutation as an \emph{instance}.
All experiments are conducted on a cluster with Intel Xeon Gold 6338 CPUs with a limit of 16GB of RAM.

It's worth noting that SCIP, along with its underlying LP
solver, is based on floating-point arithmetic. Implementing a Pseudo-Boolean
Optimization solver using a limited-precision LP-based branch-and-cut framework
comes with some technical challenges which are discussed, e.g., in
\cite{berthold2008solving, berthold2009nonlinear}. From a theoretical
standpoint, switching between reals and integers (rather than between limited and
arbitrary precision) is straightforward:

All the algorithms presented in \Cref{sec:techniques} can be naturally extended to the case of $0$--$1$ constraints with coefficients that are real numbers instead of nonnegative integers.
The Chv\'atal-Gomory procedure, MIR cutting, and coefficient tightening algorithm were originally designed for MIP with real coefficients.

However, in practice, floating-point arithmetic may cause numerical issues due to imprecise representations of real numbers and cancellation effects. 
To mitigate the risk of numeric instability, many components of SCIP, such as MIR-cut generation, utilize double-double precision arithmetic \cite{dekker1971floating}, which could be also employed in conflict analysis. Currently, for constraints generated in conflict analysis, we use the following standard techniques:
\begin{itemize}
\item We terminate conflict analysis if the coefficients of the constraints span too many orders of magnitude.
Specifically, if the quotient of the largest to smallest coefficient is large (in our implementation, $10^6$), we stop conflict analysis.
\item We remove variables from the conflict constraint if their coefficients are too small (in our implementation, $10^{-9}$), thereby relaxing the constraint slightly.
\end{itemize}
The latter threshold is a common default value for the zero tolerance in MIP solvers, and the former is a common modeling recommendation for MIP.

\subsection{Pre-Experiment: Weaken-All-At-Once vs.\ Weaken-Iteratively}

As noted earlier, the weakening rule can be applied iteratively or in a single sweep. 
In preliminary experiments, we noticed
that in almost all cases, most unassigned or true literals must be weakened to
achieve the conflict analysis invariant that the resolved constraint has a negative slack. 
Table~\ref{tab:weaken} summarizes this finding for different reduction techniques:
Over all instances and all conflict analysis calls, an average between 97.3\% (MIR) and 99.7\% (Saturation) of all literals had to be weakened.
Furthermore, for most instances both weakening variants did not lead to different execution paths.

 In this case, weakening all literals at once avoids the overhead of iterative use of cuts and expensive slack
computations. Consequently, we decided to always weaken all literals at once and apply the cut rule on the reason side
only once for the remaining experiments presented in this paper. 
\begin{table}[t]
    \centering
    \small
    \begin{tabular}{lc}
        \toprule
        Setting &   avg(\%) literals weakened\\
        \midrule
        Division   &                 98.0 \\
        Saturation &                 99.7 \\
        MIR        &                 97.3 \\
        \bottomrule
    \end{tabular} 
    \caption{Average percentage of true or unassigned literals that should be weakened to preserve the conflict analysis invariant. This experiment is conducted on the test set with 3 random seeds.}
        \label{tab:weaken}
\end{table}

\subsection{Main Experiments: Comparing Different Reasoning Techniques}

In the following, we compare all different reduction techniques from Section~\ref{sec:techniques} to \scip without any conflict analysis.

In our comparisons, we report for each technique the number of optimally solved instances, as well as the shifted geometric means of the number of processed nodes and the CPU time in seconds. The shifted geometric mean, a standard performance aggregator in the MIP literature, of the values $t_1,\dots, t_n$ is defined as
\begin{equation}
    \label{eq:sgm}
    \left(\prod_{i=1}^n (t_i+s)\right)^{1/n} - s,
\end{equation}
for some $s>0$.
We set the shift $s$ to 1 second for the CPU time and to 100 nodes for the number of nodes.
Our base of comparison is \scip without conflict analysis (``No
Conflicts''). We report absolute values for the shifted means, and also quotients comparing them to our base setting.
A factor below 1 means that a setting was faster (or needed less nodes), and a factor greater than 1 means that it was detrimental.

In \Cref{tab:results} we report the results of our experiments. The table is split in four parts.
We show results for ``all'' instances, as well as for three subsets of instances:
(i) instances that are
``affected'' by conflict analysis, hence where the execution path of at least one setting differs from the others, (ii) `` $[100,\text{limit}]$ '' instances, which take at least 100 seconds to solve to optimality or hit the time limit and
(iii) ``all-optimal'', which are instances solved by
all settings. Note that the number of nodes can only be fairly compared on the  ``all-optimal'' subset, since the number of nodes when hitting a time limit is hard to interpret and hard to aggregate with the same statistics on instances that are solved to optimality.

The variant of \scip with clausal conflict analysis is
referred to as ``Clausal-CA''.
For a fair comparison of the different strategies, we disabled the upgrading of constraints to specialized types, \ie all generated conflicts are treated as linear constraints, and further only generated one conflict per call.
Conversely, we accept PB reasoning conflicts only if the number of nonzeros is less than 15\% of the original problem variables, as in the default clausal implementation in \scip.
Our preliminary experiments confirmed that in our implementation, it is indeed detrimental to accept too-long conflicts.
We did, however, add a fallback strategy, of applying weakening on the
conflict constraint if the constraints are too long. This happens for about 9\% of the instances.

\begin{table}[t]
   \centering
   \small
    \begin{tabular}{llrrrrrrrrrrrrrrrr}
    \toprule
     & Setting & solved & time(s)&\# nodes&time quot& nodes quot\\
    \midrule
all(975) & No Conflicts & 394 & 656.75 & 784 & 1.0 & 1.0\\
 & Clausal-CA & 405 & 603.55 & 682 & 0.92 & 0.87\\
 & Division & 419 & 601.4 & 683 & 0.92 & 0.87\\
 & MIR &\color{blue} 420 & \color{blue}599.37 & 677 & 0.91 & 0.86\\
 & Saturation & 418 & 599.76 & 692 & 0.91 & 0.88\\
 \midrule
affected(295) & No Conflicts & 259 & 160.46 & 1096 & 1.0 & 1.0\\
 & Clausal-CA & 270 & 122.64 & 776 & 0.76 & 0.71\\
 & Division & 284 & 119.24 & 707 & 0.74 & 0.65\\
 & MIR & \color{blue}285 & 118.29 & 700 & 0.74 & 0.64\\
 & Saturation & 283 & \color{blue}118.09 & 735 & 0.74 & 0.67\\
 \midrule
$[100,\text{limit}](218)$ & No Conflicts & 182 & 667.14 & 2056 & 1.0 & 1.0\\
 & Clausal-CA & 193 & 486.45 & 1466 & 0.73 & 0.71\\
 & Division & 207 & 486.23 & 1345 & 0.73 & 0.65\\
 & MIR &\color{blue} 208 & \color{blue}485.26 & \color{blue}1336 & 0.73 & 0.65\\
 & Saturation & 206 & 491.98 & 1428 & 0.74 & 0.69\\
\midrule
 all-optimal(374) & No Conflicts & 374 & 46.16 & 320 & 1.0 & 1.0\\
 & Clausal-CA & 374 & 40.58 & 259 & 0.88 & 0.81\\
 & Division & 374 & 40.75 & 244 & 0.88 & 0.77\\
 & MIR & 374 & 40.40 & \color{blue} 241 & 0.88 & 0.75\\
 & Saturation & 374 & \color{blue}40.24 & 246 & 0.87 & 0.77\\
 \bottomrule\end{tabular}
\caption{Main results}
\label{tab:results}
\end{table}

We observe that all conflict analysis variants solved more instances than SCIP without conflict analysis, needed significantly less nodes on the all-optimal set, and less time on all four instance sets.
Note that on average, the time spent in conflict analysis is only about 0.1\% of the total run time.
The three PB conflict analysis variants could solve more instances than the clausal variant, and needed significantly less nodes. The difference in time was less pronounced.

The performance of the PB conflict analysis variants is quite similar in all three cases.
Nevertheless, MIR-based reduction could solve the most instances and needed the least nodes on the all-optimal set.
When looking at the seemingly identical time-wise performance in more detail, it turns out that MIR also slightly improves on the other settings in this measure.
There are 104 instances for which the path differs between Saturation-based resolution and MIR-based resolution and MIR was on average 1.1\% faster on those.
There are 86 instances for which the path differs between Division-based resolution and MIR-based resolution and MIR was on average 3.6\% faster on those.
Consequently, we decided to concentrate on MIR-based resolution for our next statistic.
\begin{table}[ht]
    \centering
    \small
    \begin{tabular}{lccc}
        \toprule
        Setting &  mean \# conflicts &  avg \% prop.\ conflicts & avg \# literals\\
        \midrule
	Clausal-CA & 290.77 &         34.54 &  82.45 \\
	MIR & 169.61 &         58.54 &  80.20 \\
        \bottomrule
    \end{tabular} 
    \caption{Shifted geometric mean of number of conflicts, average percentage of conflict constraints that propagate at least once
     and average length of learned conflicts.}
        \label{tab:useful}
\end{table}

Ultimately, the purpose of conflict constraints is to restrict the future search space by propagating literal assignments and pruning the search tree.
Hence we analyzed how many conflicts each of the methods generates in shifted geometric mean, how large these conflicts are on average, and how many of them lead to propagations down the road.
Table~\ref{tab:useful} shows the results on the set of all instances that have a search tree of at least 100 nodes (to get a decent chance of conflict generation and propagation) and for which at least one conflict was generated with one of the methods. We consider only instances where the two settings have the same execution path. We observe that our MIR-based conflict analysis generated about a third less conflicts, but at the same time, they are much more likely to propagate: for the classic clausal conflict analysis of SCIP, about a third of the generated conflicts are used for propagation later on, while for our MIR-based variant, slightly more than half (58.54\%) of all conflicts propagate at least once.
At the same time, MIR-based conflicts are about the same size as clausal conflicts.

At first glance, this might appear as a contradiction, given that, as a rule of thumb, shorter conflicts tend to propagate more often and one might expect similar-sized conflicts to be similarly likely to propagate.
Note, however, that the conflicts are of a quite different nature in the two cases. 
On the one hand, clausal conflicts are always logic clauses 
that only propagate when all but one literal are assigned.
On the other hand, MIR-based conflicts are general pseudo-Boolean constraints, which might propagate some assignments (of literals with large coefficients) even when a majority of literals are still unassigned.
This goes nicely together with the above observation that the reduction in the number of nodes is more pronounced than the reduction in runtime.
As a final remark, integrating PB conflict analysis in a production-grade MIP solver would require
substantially more work, but should also be expected to provide
substantial further improvements
measured in wallclock time.

 \section{Conclusion}
\label{sec:conclusion}

In this work, we study how to integrate pseudo-Boolean conflict
analysis for \mbox{$0$--$1$} integer linear programs into a MIP
solving framework.  In contrast to standard MIP conflict analysis, the
pseudo-Boolean method operates directly on the linear constraints,
rather than on clauses extracted from these constraints, and
this makes it exponentially stronger in terms of reasoning power.
Viewing PB conflict analysis from a MIP
perspective is also helpful since it provides a view of the
algorithm as a sequence of linear combinations and cuts, and we use this
to strengthen the pseudo-Boolean conflict analysis further by developing a
new conflict analysis method using the powerful mixed integer rounding
(MIR) cuts.

We have made a first proof-of-concept implementation of our new
pseudo-Boolean conflict analysis method, 
as well as methods from the PB literature based on
saturation~\cite{LP10Sat4j}
and division~\cite{EN18RoundingSat},
in the  open-source MIP solver \solver{SCIP}, 
and have run experiments on
\mbox{$0$--$1$} ILP instances from \miplib2017
comparing the different methods with each other and with standard
clause-based MIP conflict analysis.
We find that solving \mbox{$0$--$1$} ILPs
with MIR-based pseudo-Boolean conflict analysis performs better than
other methods, not only in the sense that it reduces the size of the
search tree, but also in that our implementation can  beat the
highly optimized MIP conflict analysis currently used in \scip in terms of
actual running time. In our opinion, this demonstrates convincingly that 
pseudo-Boolean conflict analysis in MIP is a research direction that
should be worth pursuing further, and that similar proof-of-concept
studies could also be relevant to investigate for other combinatorial
solving paradigms such as constraint programming.

As already noted above, an obvious direction of future work is to
provide a more carefully engineered version of pseudo-Boolean conflict
analysis that could deliver more fully on the potential for
improved performance identified by our experiments.
In addition to optimizing the existing code, however,
it would be valuable to develop a better understanding of how and
why the conflict analysis works and of ways in which the reasoning could be improved.

Pseudo-Boolean conflict analysis alternates between weakening constraints (to
eliminate seemingly less relevant variables) and strenghtening them by
applying cut rules (to get tighter propagation on the variables that
remain). The  interplay between these two operations is quite
poorly understood even for pseudo-Boolean solvers, and so 
both PB solvers and MIP solvers could gain from a 
careful study of how to strike the right balance.
Since PB conflict analysis can be performed with several
different reduction methods, and since different reduction methods
can be employed independently in consecutive steps in one and the same conflict
analysis, it would also be good to be able to assess the quality of
constraints derived during conflict 
analysis, so as to select the most promising candidate at each step 
to pass on to the next step in the conflict analysis.

Arguably the most interesting research question, though,  is whether pseudo-Boolean
conflict analysis could be extended beyond
\mbox{$0$--$1$} ILPs to 
\mbox{$0$--$1$} mixed linear problems,
and/or to general integer linear programs.
It is worth noting that the latter  has been attempted
in~\cite{JdM13Cutting,Nieuwenhuis14IntSat},
but so far with quite limited success.
It is clear that the algorithms presented in this paper \emph{cannot}
work for 
\mbox{$0$--$1$} mixed LPs
or general ILPs
if generalized in the obvious, naive way,
and so additional, new ideas will be needed.

   \section*{Acknowledgements}
The work for this article has been conducted within the Research Campus Modal funded by the German Federal Ministry of Education and Research (BMBF grant numbers 05M14ZAM, 05M20ZBM).
Jakob Nordström was supported by the Swedish Research Council grant \mbox{2016-00782} and the Independent Research Fund Denmark grant \mbox{9040-00389B}.

Part of this work was carried out while some of the authors participated
in the extended reunion for the program
\emph{Satisfiability: Theory, Practice, and Beyond}
at the Simons Institute for the Theory of Computing at UC Berkeley
in the spring of 2023.
This work has also benefited greatly from discussions during the
Dagstuhl Seminar 22411
\emph{Theory and Practice of SAT and Combinatorial Solving}.

\bibliographystyle{plainurl}
\bibliography{bibliography,refArticles,refBooks}

\end{document}